\label{key}
\documentclass[12pt, reqno]{amsart}
\usepackage{amssymb, amsthm, amsmath, amsfonts}
\usepackage{array, epsfig}
\usepackage{bbm}
\usepackage{hyperref}
\usepackage[numbers,square]{natbib}
\usepackage{color}
\usepackage{graphicx}

  \evensidemargin
\oddsidemargin
\newtheorem{theorem}{Theorem}[section]

\newtheorem{lemma}[theorem]{Lemma}
\newtheorem{corollary}[theorem]{Corollary}

\newcommand{\E}{\mathbb{E}\,}
\renewcommand{\P}{\mathbb{P}}
\newcommand{\alg}{\mathbb{A}}
\newcommand{\R}{\mathbb{R}}
\newcommand{\Z}{\mathbb{Z}}
\newcommand{\C}{\mathbb{C}}
\newcommand{\N}{\mathbb{N}}

\newcommand{\bx}{\mathbf{x}}
\newcommand{\bw}{\mathbf{\lambda}}

\newcommand{\dd}{{\rm d}}

\newcommand{\B}{\mathbb{B}}
\newcommand{\T}{\mathbb{T}}
\newcommand{\iu}{\mathrm{i}\mkern1mu}

\newcommand{\h}{Q} 

\newcommand{\Vol}{\mathop{\mathrm{Vol}}\nolimits}

\newcommand{\eqdistr}{\stackrel{d}{=}}

\newcommand{\tc}{}

\title{Distribution of complex algebraic numbers \\on the unit circle}

\keywords{Bombieri norm, distribution of algebraic numbers, integral polynomials, random trigonometric polynomials, real zeros }
\subjclass[2010]{Primary, 11N45; secondary, 11C08, 30C15} 
\thanks{\tc{The research of the first author was supported by  SFB 1283 and the research of the second author was supported by  IRTG 2235 at Bielefeld University (Germany). }}

\author{Friedrich G\"otze}
 \address{Friedrich G\"otze: Faculty of Mathematics,
 Bielefeld University,
 P. O. Box 10 01 31,
 33501 Bielefeld, Germany}
 \email{goetze@math.uni-bielefeld.de}

\author{Anna Gusakova}
\address{Anna Gusakova: Faculty of Mathematics,
 Bielefeld University,
 P. O. Box 10 01 31,
 33501 Bielefeld, Germany}
\email{agusakov@math.uni-bielefeld.de}

\author{Zakhar Kabluchko}
\address{Zakhar Kabluchko: M\"unster  University,   
Orl\'eans-Ring  10,   
48149  M\"unster,   Germany}
\email{zakhar.kabluchko@uni-muenster.de}

\author{Dmitry Zaporozhets}
\address{Dmitry Zaporozhets: St.\ Petersburg Department of Steklov Mathematical Institute,
Fontanka~27,
191023 St.\ Petersburg, Russia}
\email{zap1979@gmail.com}

\begin{document}

\begin{abstract}
For $-\pi\leq\beta_1<\beta_2\leq\pi$ denote by $\Phi_{\beta_1,\beta_2}(Q)$ the  number of  algebraic numbers on the unit circle with arguments in $[\beta_1,\beta_2]$ of degree $2m$ and \tc{with  elliptic height} at most $Q$. We show that 
\[
\Phi_{\beta_1,\beta_2}(Q)=Q^{m+1}\int\limits_{\beta_1}^{\beta_2}{p(t)}\,{\rm d}t+O\left(\h^m\,\log Q\right),\quad Q\to\infty,
\]
where $p(t)$  coincides up to a constant factor with the density of the roots of some random trigonometric polynomial. This density is calculated explicitly using the Edelman--Kostlan formula.
\end{abstract}

\maketitle

\section{Introduction}

The problem of \tc{distribution of real and complex algebraic numbers has been studied intensively  in the last two decades
and turned out to be closely related with the distribution of  roots of random polynomials. We first give  a  brief review of recent results 
in this area before  describing the problem considered in this paper.}

Let $\alg$ denote the field of (all) algebraic numbers over $\mathbb{Q}$ and let $\alg_n$ denote the set of algebraic numbers of degree $n \in \N$.
\tc{Note that the set $\alg_n$ is countable and any open subset  of $\R$ or $\C$ contains already} infinitely many algebraic numbers. In order to study the distribution of \tc{those} algebraic numbers, we need to
\tc{select finite (ordered) subsets of }
$\alg_n$. To this end, we consider a \emph{height function} $h:\alg\to\R_+$ such that for any $n\in\N$ and $\h>0$ there are only finitely many algebraic numbers
$\alpha$ of degree~$n$ with $h(\alpha)\leq \h$. Note that 
\tc{one usually
	requires (which  we will always assume) that} $h(\alpha')=h(\alpha)$ for all conjugates of~$\alpha$. 

A natural question is to \tc{determine  the asymptotics of the number of $\alpha\in\alg_n$ lying}  in a given subset of $\R$ or $\C$ such that $h(\alpha)\leq\h$ and degree $n$ is \emph{fixed} {as $\h\to\infty$}. 
\tc{This problem has been 
 studied by Masser and Vaaler~\cite{MV08} with height function
 being the Mahler measure}. Later Koleda~\cite{dK14} \tc{determined the} asymptotic number of  real algebraic numbers ordered by the na\"{\i}ve height; the same result for complex algebraic numbers was obtained in~\cite{GKZ15}. A generalization of the na\"{\i}ve height, the weighted $l_p$-norm (which also generalizes the length, the \tc{Euclidean} norm, and the Bombieri $p$-norm), was considered in~\cite{GKZ17}. Another interesting example of heights -- \tc{namely} the house of an algebraic number -- was studied in~\cite{CH17}, \tc{which studies the distribution of Perron numbers.}


In the \tc{present note} we would like to study the distribution of  algebraic numbers on the \emph{unit circle}. Although our methods work for any weighted $l_p$-norm (including \tc{the} na\"{\i}ve height), we consider the weighted \tc{Euclidean} (or elliptic) height only: this case corresponds to the simplest asymptotic distribution formula. 

The paper is organized as follows. In the next section, we formulate our main result Theorem~\ref{0001} and give some corollaries. The proof of Theorem~\ref{0001} is given in Section~\ref{0021}. Section~\ref{2349} contains the proof of some auxiliary statements and the corollaries are proved in Section~ \ref{0023}.

\section{Basic notation and main result}\label{0913}
Given a polynomial $P(t):=a_nt^n+\ldots+a_1t+a_0$ and a vector of positive weights $\bw=(\lambda_0, \lambda_1,\dots,\lambda_n)$ define the elliptic height  of $P$ as
\[
h_{\bw}(P):=\left(\sum\limits_{k=0}^{n}\frac{a_k^2}{\lambda_k^2}\right)^{1/2}.
\]

Let $\mathcal{P}(\h)$ denote the class of integral polynomials (that is, with integer coefficients) of degree $n$ and with height $h_{\bw}$ at most $\h$:
\[
\mathcal{P}(\h):=\{P\in\Z[t]\colon\deg P=n, \, h_{\bw}(P)\leq\h\}.
\]
We say that an integral polynomial  is \emph{prime}, if it is irreducible over $\mathbb{Q}$, primitive (with co-prime coefficients), and its leading coefficient is positive. Denote by
$\mathcal{P}^{*}(\h)$ the class of prime polynomials from $\mathcal{P}(\h)$:
\[
\mathcal{P}^{*}(\h):=\{P\in\mathcal{P}(\h)\colon P \text{ is prime}\}.
\]

The \emph{minimal polynomial} of an algebraic number $\alpha$ is the uniquely defined prime polynomial $P\in\Z[t]$ satisfying $P(\alpha)=0$. Then, the  elliptic height of  $\alpha$ is defined as {the elliptic height of its minimal polynomial:}
\[
h_{\bw}(\alpha):=h_{\bw}(P).
\]
The other roots of $P$ are called algebraic conjugates of $\alpha$.

We aim to investigate the asymptotic behavior of the algebraic numbers on the unit circle $\T\subset\mathbb{C}$:
\[
\T:=\{z\in\C:|z|=1\}.
\]

We start with the simple observation that any algebraic number on $\T$ has even degree and the coefficients of its minimal polynomial posses some symmetry property. Indeed, consider an algebraic number $\alpha\in \T$ with the minimal polynomial
\[
P(t)=a_nt^n+\ldots +a_1t+a_0.
\] 
Since the coefficients of $P$ are real, $\bar\alpha$ is also a root of $P$. Moreover, $\alpha\in \T$ implies $\bar{\alpha}=1/\alpha$. Thus,
\[
P(\alpha)=P\left(\frac{1}{\alpha}\right)=0,
\]
{which means that $\alpha$ is a root of the polynomial}
\[
\tilde{P}(t):=t^nP(t^{-1})=a_0t^n+\ldots +a_{n-1}t+a_n.
\] 
{Hence, by definition of minimal polynomial} we \tc{conclude} that $P$ is a multiple of $\tilde P$ \tc{which leaves us with two possibilities only:}
 $P\equiv-\tilde P$ or $P\equiv\tilde P$. The former one would imply that $1$ is a root of $P$ which is impossible due to its irreducibility. Therefore $P\equiv\tilde P$ which means
\[
a_k=a_{n-k},\quad 0\leq k\leq n.
\]
For odd $n$, this condition implies that $-1$ is a root of $P$ which, again, contradicts  
with its irreducibility. Thus from now on we consider only {algebraic numbers of} even degree
\[
n=2m,\quad m\geq 1.
\]
The number $m$ \tc{will stay} 
throughout the 
paper.

Now we are ready to formulate our main result. For 
\begin{equation}\label{2325}
-\pi\leq\beta_1<\beta_2\leq\pi
\end{equation}
denote by $\Phi_{\beta_1,\beta_2}(Q)$ the  number of  algebraic numbers \tc{of degree $n=2m$}  on the unit circle with arguments in $[\beta_1,\beta_2]$  and with \tc{elliptic height} at most $\h$:
\[
\Phi_{\beta_1,\beta_2}(Q):=\#\left\{\theta\in[\beta_1,\beta_2]\colon e^{i\theta}\in\alg_{2m}, h_{\bw}\left(e^{i\theta}\right)\leq\h\right\}.
\]
Denote by $\zeta(\cdot)$  the Riemann zeta function and by $\B^{m+1}$  the $(m+1)$-dimensional unit ball of volume $\Vol(\B^{m+1})$.
\begin{theorem}\label{0001}
For any fixed symmetric vector of positive weights $\bw=(\lambda_0,\dots,\lambda_m,\dots,\lambda_{0})$   we have
\[
\Phi_{\beta_1,\beta_2}(Q)=\frac{\Vol(\B^{m+1})}{2^{m/2+1}\pi\zeta(m+1)}\lambda_0\dots\lambda_m\h^{m+1}\int\limits_{\beta_1}^{\beta_2}{p_{\bw}(t)}\,\dd t+O\left(\h^m\,\log\h\right),\quad Q\to\infty,
\]
where 
\begin{equation}\label{eq15}
p_{\bw}(t):=\left[\frac{\partial^2}{\partial x\partial y}\log\left(\frac{\lambda_m^{2}}{2}+\sum\limits_{k=1}^{m}\lambda_{m-k}^{2}\,\cos(kx)\cos(ky)\right)\Big|_{x=y=t}\right]^{1/2}.
\end{equation}
The implicit constant  depends on $\bw$ and $m$ only and the $\log\h$ factor  can be omitted for $m\geq3$.

\end{theorem}
In this general form, the limit density $p_{\bw}$ is difficult to analyze. However, there are some examples of $\lambda$ where {\eqref{eq15} can be simplified}.
The first one is the \tc{vector of Bombieri 2-norm weights}.


\begin{corollary}\label{0003}
For  $\lambda_k=\sqrt{\binom{2m}{k}},k=0,\dots,2m,$ we have
	\[
	p_{\bw}(t)=\sqrt{\frac{m}{2}}\cdot\frac{\left|\sin t\right|\,\sqrt{\sum\limits_{k=0}^{2m-2}\left(\cos t\right)^{2k}+(2m-1)\left(\cos t\right)^{2m-2}}}{\left(\cos t\right)^{2m}+1}.
	\]
\end{corollary}

\begin{figure}[h]
\begin{minipage}[h]{0.49\linewidth}
\center{\includegraphics[width=0.8\linewidth]{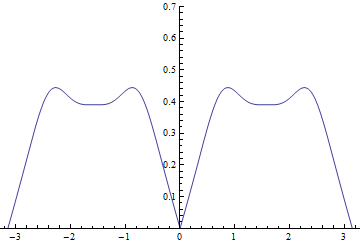} \\a}
\end{minipage}
\hfill
\begin{minipage}[h]{0.49\linewidth}
\center{\includegraphics[width=0.8\linewidth]{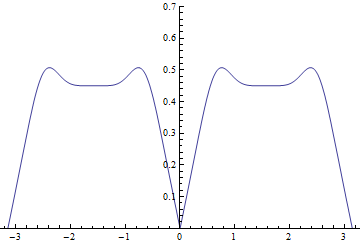} \\b}
\end{minipage}
\hfill
\hfill
\begin{minipage}[h]{0.49\linewidth}
\center{\includegraphics[width=0.8\linewidth]{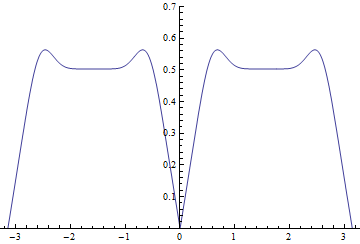} \\c}
\end{minipage}
\hfill
\begin{minipage}[h]{0.49\linewidth}
\center{\includegraphics[width=0.8\linewidth]{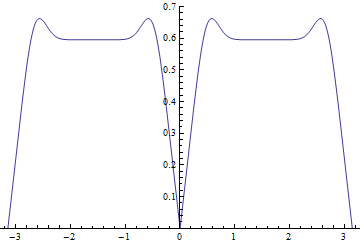} \\d}
\end{minipage}
\caption{A plot of  $p_{\bw}(t)$  with $\bw=\left(\binom{2m}{k}^{1/2}\right)_{k=0}^n$: a) $m=3$; b) $m=4$; c) $m=5$; d) $m=7$.}
\label{ris:image1}
\end{figure}

The second example is the Euclidean height.
\begin{corollary}\label{0004}
For $\bw=\left(1,\ldots,1\right)$ we have
\begin{align*}
p_{\bw}(t)&=\left(b_m+\frac{\sin(b_mt)}{\sin t}\right)^{-1}\cdot\Big(\frac{b_m\sin(b_mt)}{2(\sin t)^3}-\frac{b_m^2}{2}\frac{\cos(b_mt)\cos t}{(\sin t)^2}\\
&+\frac{\left(\sin(b_mt)\right)^2}{4(\sin t)^4}-\frac{b_m^3+2b_m}{6}\frac{\sin(b_mt)}{\sin t}-\frac{b_m^2}{4(\sin t)^2}+\frac{(m^2+m)b_m^2}{3}\Big)^{1/2},
\end{align*}
where $b_m=2m+1$.
\end{corollary}

\begin{figure}[h]
\begin{minipage}[h]{0.49\linewidth}
\center{\includegraphics[width=0.8\linewidth]{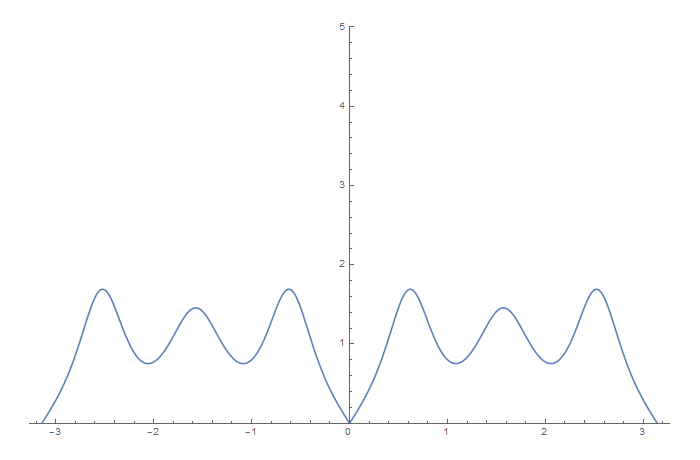} \\a}
\end{minipage}
\hfill
\begin{minipage}[h]{0.49\linewidth}
\center{\includegraphics[width=0.8\linewidth]{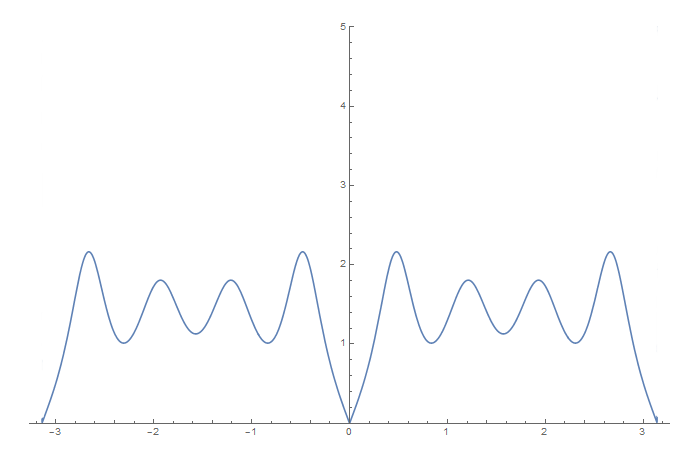} \\b}
\end{minipage}
\hfill
\hfill
\begin{minipage}[h]{0.49\linewidth}
\center{\includegraphics[width=0.8\linewidth]{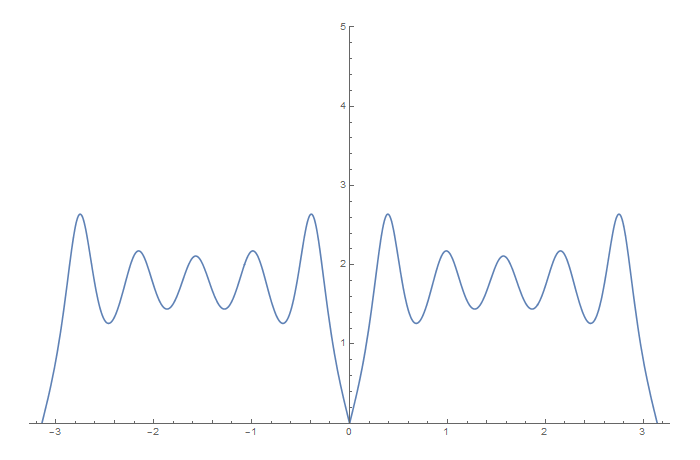} \\c}
\end{minipage}
\hfill
\begin{minipage}[h]{0.49\linewidth}
\center{\includegraphics[width=0.8\linewidth]{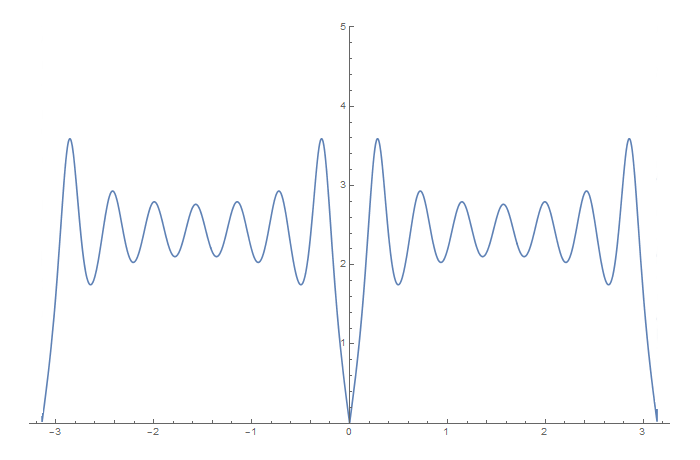} \\d}
\end{minipage}
\caption{A plot of  $p_{\bw}(t)$ with  $\bw=\left(1,\ldots,1\right)$: a) $m=3$; b) $m=4$; c) $m=5$; d) $m=7$.}
\label{ris:image2}
\end{figure}

\bigskip

As was mentioned above, the results on \tc{the} distribution of the roots of random algebraic polynomials were used in~\cite{GKZ15},~\cite{GKZ17} to study the asymptotic behavior of algebraic numbers in a domain of the complex plane. In this paper we show that the distribution of algebraic numbers on the unit circle 
{can be} described in terms of random \emph{trigonometric} polynomials. Namely, 
{one can easily see} from the proof of Theorem~\ref{0001} that
\[
\E \mu_F([\beta_1,\beta_2])=\frac{1}{\pi}\int\limits_{\beta_1}^{\beta_2}\rho_\lambda(t) \dd t,
\]
{where  
\[
F(\theta):=\frac{\lambda_m}2\eta_m+\frac{1}{\sqrt2}\sum\limits_{k=1}^m \lambda_{m-k}\eta_{m-k}\,\cos(k\theta)
\]
is a random trigonometric polynomial with coefficients $\eta_0,\ldots, \eta_m$ being i.i.d. real-valued  standard Gaussian variables and $\mu_F([\beta_1,\beta_2])$ denotes the number of the  real roots of $F$ lying in the interval $[\beta_1,\beta_2]$.}
Thus, up to a factor $\pi$, the limit density $p_{\bw}$ coincides with the density of the roots of $F$.

%

\section{Proof of Theorem~\ref{0001}}\label{0021}
As \tc{argued} in Section~\ref{0913} we can restrict our attention to the following subclass of symmetric polynomials of even degree $n=2m$:
\[
\mathcal{P}_{\mathrm{sym}}(Q):=\bigg\{P\in\mathcal{P}(Q)\colon P(t)=\sum\limits_{k=0}^{2m}a_kt^k,\,a_k=a_{2m-k}\bigg\}.
\]
Moreover let $\mathcal{P}^*_{\mathrm{sym}}(Q)$ denote the subclass of prime polynomials from $\mathcal{P}_{\mathrm{sym}}(Q)$:
\[
\mathcal{P}^*_{\mathrm{sym}}(Q):=\mathcal{P}_{\mathrm{sym}}(Q)\cap\mathcal{P}^*(Q).
\]
Given a function $g:\mathbb{C}\rightarrow\C$ and a  subset $B\subset\mathbb{C}$ denote by $\mu_g(B)$ the number of the roots of  $g$ lying in $B$. For $-\pi\leq\beta_1<\beta_2\leq\pi$ denote by $\T_{\beta_1,\beta_2}$ the arc
\[
\T_{\beta_1,\beta_2}:=\left\{z\in\T\colon \arg z\in[\beta_1,\beta_2]\right\}.
\]
Since $\mathcal{P}_{\mathrm{sym}}^{*}(\h)$ is the set of all minimal polynomials for algebraic numbers $\alpha\in \T$ with $\deg\alpha=2m$ and $h_{\bw}\left(\alpha\right)\leq \h$, we have
\[
\Phi_{\beta_1,\beta_2}(Q)=\sum\limits_{P\in\mathcal{P}_{\mathrm{sym}}^{*}(\h)}\mu_P(\T_{\beta_1,\beta_2})
\]
and since $\mu_P(\T_{\beta_1,\beta_2})\leq 2m$, we can write
\begin{equation}\label{1950}
\Phi_{\beta_1,\beta_2}(Q)=\sum\limits_{l=0}^{2m}l\,\#\{P\in\mathcal{P}_{\mathrm{sym}}^{*}(\h):\,\mu_P(\T_{\beta_1,\beta_2})=l\}.
\end{equation}

Thus our aim is to estimate the number of the prime symmetric polynomials \tc{having
a prescribed number of roots} in a {given} set. 
\tc{Since these polynomials can be 
	identified 
	with the vectors of their coefficients}, our first step is to find a way of counting integer points in multidimensional regions.

For a Borel set $A\subset\R^d$ denote by $\gamma(A)$ (resp.  $\gamma^*(A)$) the number of points in $A$ with integer  (resp. co-prime integer) coordinates. For a real number $r$ {define the dilated set} 
\[
rA:= \{r\bx : \bx\in A \}.
\]


For our purposes we would like to know the asymptotic behavior of the quantity $\gamma^*(\h A)$ as $\h\rightarrow\infty$. 
\tc{In order to get a good estimate one should} impose some regularity conditions on the boundary $\partial A$ of $A$. According to~\cite[Definition 2.2]{mW12}, we say that  $\partial A\subset\R^d$ is of Lipschitz class $(M,L)$  if there exist $M$ maps
$\phi_1, \dots, \phi_M : [0, 1]^{d-1} \to\R^d$ satisfying a Lipschitz condition
\[
|\phi_i(x) - \phi_i(y)| \le L|x - y| \text{ for } x, y \in [0, 1]^{d-1}, \quad  i = 1, \dots,M,
\]
such that $\partial A$ is covered by the images of the maps $\phi_i$. 

The next lemma 
{provides} the asymptotic of $\gamma^*(\h A)$ as $\h\rightarrow\infty$.

\begin{lemma}\label{lm-prim-pnt}
	Consider a bounded region $A\subset\mathbb{R}^d$, $d\geq2$, such that the boundary $\partial A$ of $A$ is of Lipschitz class $(M,L)$. 
	Then $A$ is Lebesgue measurable and 
	\[
	\left|\frac{\lambda^*(QA)}{\h^d}-\frac{\Vol(A)}{\zeta(d)}\right|\leq C\frac{\log^{\chi(d)}\h}{\h},
	\]
	where  $C$ depends on $d,L,M$ only and 
	\begin{equation}\label{2151}
\chi(d):=\left\{
	\begin{array}{ll}
	1, & d=2; \\
	0, & \text{otherwise}.
	\end{array}
	\right.
	\end{equation}
\end{lemma}
\begin{proof}
	See, e.g.,~\cite[Lemma~6.3]{GKZ17}.
\end{proof}

%
%

%
%
%
%
%
%
%
%
%
%
%
%

For $l=0,1,\dots, 2m$ denote by $A_l\subset\R^{m+1}$ the set of points $(a_0,\ldots,a_m)$ such that the polynomial $P(t)=a_0t^{2m}+\ldots+a_mt^{m}+\ldots+a_0$ satisfies $\mu_P(\T_{\beta_1,\beta_2})=l$ and $h_{\bw}\left(P\right)\leq 1$. The latter condition is equivalent to $(a_0,\ldots,a_m)\in\mathcal{E}_\bw$, where $\mathcal{E}_\bw\subset\R^{m+1}$ is the ellipsoid defined by
\[
\mathcal{E}_\bw:=\bigg\{(a_0,\ldots,a_m)\in\R^{m+1}:\frac{a_m^2}{\lambda_m^2}+2\sum\limits_{k=0}^{m-1}\frac{a_k^2}{\lambda_k^2}\leq1\bigg\}.
\]
Note that
\[
\Vol(\mathcal{E}_\bw)=2^{-m/2}\lambda_0\dots\lambda_m\Vol(\B^{m+1}).
\]
Then by definition of a primitive polynomial we have
\[
\gamma^*(\h A_l)=\#\{P\in\mathcal{P}_{\mathrm{sym}}(\h):P\text{ is primitive},\,\mu_P(\T_{\beta_1,\beta_2})=l\},
\]
which implies
\begin{equation}\label{eq12}
\Big|\frac12\,\gamma^*(\h\,A_l)-\#\{P\in\mathcal{P}_{\mathrm{sym}}^{*}(\h):\,\mu_P(\T_{\beta_1,\beta_2})=l\}\Big|\leq r(\h).
\end{equation}
Here, $r(\h)$ denotes the number of the   polynomials reducible in  $\mathcal{P}_{\mathrm{sym}}(\h)$ (i.e., {can be represented} as a product of two symmetric integral polynomials of positive degree). The  factor $1/2$ in~\eqref{eq12} is due to the positiveness of the leading coefficient of a prime polynomial. 

Thus, our next step is to estimate $\gamma^*(\h\,A_l)$ and $r(\h)$. The latter is established  in the following lemma.
%
\begin{lemma}\label{003}
For some $C>0$ depending on $m,\lambda$ only we have
\[
r(\h)\leq C\cdot\h^m\,\log^{\chi(m)}\h,
\]
where $\chi(\cdot)$ is defined in (\ref{2151}). 
\end{lemma}
The proof of  Lemma~\ref{003} is postponed to Section~\ref{2349}. 
{In order to} estimate $\gamma^*(\h\,A_l)$, we are going to apply Lemma~\ref{lm-prim-pnt}. 
{Before we do it}, we need to check that the boundary of $A_l$ is of Lipschitz class.
\begin{lemma}\label{000}
{For any $0\leq l\leq m$, there exist  constants $M$, $L$ depending on $m,\bw$ only such that the boundary $\partial A_l$ is of Lipschitz class $(M,L)$. }
\end{lemma}
This lemma is a slightly modified and simplified version of~\cite[Lemma~6.4]{GKZ17}. For the reader's convenience, we give a detailed proof of  Lemma~\ref{000} in Section~\ref{2349}.

Now 
{taking Lemma~\ref{000} into account} we can apply Lemma~\ref{lm-prim-pnt} to the set $A_l$ which together with~\eqref{eq12} and Lemma~\ref{003} gives
\[
\#\{P\in\mathcal{P}_{\mathrm{sym}}^{*}(\h):\,\mu_P\left(\T_{\beta_1,\beta_2}\right)=l\}=\frac{\Vol(A_l)}{2\zeta(m+1)}\,\h^{m+1}+O\left(\h^{m}\,\log^{\chi(m)}\h\right),
\]
and, due to~\eqref{1950}, {we conclude} that
\begin{equation}\label{eq13}
\Phi_{\beta_1,\beta_2}(Q)=\frac{\h^{m+1}}{2\zeta(m+1)}\,\sum\limits_{l=0}^{2m}l\,\Vol(A_l)+O\left(\h^{m}\,\log^{\chi(m)}\h\right).
\end{equation}

To compute the sum on the right-hand side of (\ref{eq13}) consider the random polynomial 
\[
G(z):=\sum_{k=0}^{m-1}\xi_k (z^k+z^{2m-k})+\xi_mz^{m},
\]
where the random vector 
\begin{equation}\label{1102}
\left(\frac{\xi_0}{\lambda_0/\sqrt2},\dots,\frac{\xi_{m-1}}{\lambda_{m-1}/\sqrt2},\frac{\xi_m}{\lambda_m}\right)
\end{equation}
 is uniformly distributed over the $(m+1)$-dimensional unit ball $\B^{m+1}$. Then, by definition of  $A_l$ and since the semi-axes of $\mathcal{E}_\bw$ are $\lambda_0/\sqrt2,\dots,\lambda_{m-1}/\sqrt2,\lambda_m$, we have
\begin{equation}\label{eq14}
\P\left[\mu_G\left(\T_{\beta_1,\beta_2}\right)=l\right]=\frac{\Vol(A_l)}{\Vol(\mathcal{E}_\bw)}.
\end{equation}
Taking  $z=e^{i\theta}\in \T$ leads to
\begin{align*}
G(z)&=\sum_{k=0}^{m-1}\xi_k (e^{ik\theta}+e^{i(2m-k)\theta})+\xi_me^{im\theta}=2e^{im\theta}\,\left(\sum_{k=0}^{m-1}\xi_k \frac{e^{-\iu(m-i)\theta}+e^{i(m-k)\theta}}{2}+\frac{\xi_m}{2}\right)\\
&=2e^{im\theta}\,\left(\sum_{k=1}^{m}\xi_{m-k} \cos\left(k\theta\right)+\frac{\xi_m}{2}\right)=:2e^{im\theta}\,\tilde F(\theta),
\end{align*}
{and, hence,}
\begin{equation}\label{2103}
\P\left[\mu_G\left(\T_{\beta_1,\beta_2}\right)=l\right]=\P\left[\mu_{\tilde F}\left(\left[\beta_1,\beta_2\right]\right)=l\right].
\end{equation}
The latter probability  is still difficult to calculate because of the dependency of the coefficients of $\tilde F$.
However, by proper normalization, which does not affects the roots, we can \tc{achieve   independence}.

Namely, let $\eta_0,\ldots, \eta_m$ be i.i.d. real-valued  standard Gaussian random variables and let $Z$ be a standard exponential random variable. It is known (see, e.g.,~\cite[Chapter~2]{FZ90}) that the random vector
\[
\frac{(\eta_0,\eta_1,\dots,\eta_m)}{\Big(\sum\limits_{k=0}^{m}\eta_k^2+Z\Big)^{1/2}}
\]
is uniformly distributed in the unit ball $\B^{m+1}$, and, hence, has  the same distribution as~\eqref{1102}. Thus,
\[
\frac{\Big(\frac{\lambda_0\eta_0}{\sqrt2},\dots,\frac{\lambda_{m-1}\eta_{m-1}}{\sqrt2},\lambda_m\eta_m\Big)}{\Big(\sum\limits_{k=0}^{m}\eta_k^2+Z\Big)^{1/2}}\eqdistr (\xi_0,\dots,\xi_m).
\]
{Using the fact that} dividing a polynomial by a non-zero constant does not affect its roots, we conclude that the polynomials $\tilde F(\theta)$ and
\[
F(\theta):=\frac{\lambda_m}2\eta_m+\frac{1}{\sqrt2}\sum\limits_{k=1}^m \lambda_{m-k}\eta_{m-k}\,\cos(k\theta)
\]
have the same distribution of the roots and
\[
\P\left[\mu_{\tilde F}\left(\left[\beta_1,\beta_2\right]\right)=l\right]=\P\left[\mu_{ F}\left(\left[\beta_1,\beta_2\right]\right)=l\right].
\]
 Combining this with~\eqref{eq14} and~\eqref{2103} gives
\begin{equation}\label{0918}
\sum\limits_{l=0}^{2m}l\,\Vol(A_l)=\Vol(\mathcal{E}_\bw)\sum\limits_{l=0}^{2m} l\,\P[\mu_{ F}\left(\left[\beta_1,\beta_2\right]\right)=l]
=\Vol(\mathcal{E}_\bw)\E [\mu_{ F}\left(\left[\beta_1,\beta_2\right]\right)].
\end{equation}
Finally, applying the Edelman--Kostlan formula (see~\cite[Theorem~3.1]{EK95}) to the random function $F$
leads to
\[
\E [\mu_{\tilde F}\left(\left[\beta_1,\beta_2\right]\right)] =\frac{1}{\pi}\,\int\limits_{\beta_1}^{\beta_2}{p_{\bw}(t)}\,dt,
\]
where
\[
p_{\bw}(t)=\left[\frac{\partial^2}{\partial x\partial y}\log\left(\frac{\lambda_m^{2}}{2}+\sum\limits_{k=1}^{m}\lambda_{m-k}^{2}\,\cos(kx)\cos(ky)\right)\Big|_{x=y=t}\right]^{1/2},
\]
which together with~\eqref{eq13} and~\eqref{0918} finishes the proof.

\section{Proofs of Lemma~\ref{003} and Lemma~\ref{000} }\label{2349}
\subsection{Proof of Lemma~\ref{003}}
The proof follows {the scheme described in}~\cite{Ku09}.

For non-negative functions $f, g$ we write $f\ll g$ if there exists a non-negative constant $C$ depending on $m,\lambda$ only, such that $f\leq Cg$. Given a polynomial $P(t)=a_nt^n+\ldots+a_1t+a_0$ denote by $H(P)$ its na\"ive height: $H(P):=\max_{0\leq i\leq n}|a_i|$. Our first step is to prove the lemma for the na\"ive height instead of $h_\lambda$.

Let $\tilde r(\h)$ be the number of symmetric integral polynomials of degree $2m$ and with the na\"ive height at most $\h$
{which can be represented} as a product of two symmetric polynomials of positive degrees. Denote by  $s(\h)$  the number of pairs $(P_1$, $P_2)$ of symmetric integral polynomials  such that $\deg P_1 +\deg P_2 = 2m$ and 
\[
H(P_1)H(P_2)\leq \h.
\]
{It is easy to see that} the number of  symmetric integral polynomials of degree $2k$ and with the na\"ive height  $\h$ is of order $O(\h^k)$. Hence,
\[
s(\h)\ll \sum\limits_{k=1}^{m-1}\sum\limits_{\substack{x,y\in\Z, x,y \ge 1,\\ xy\leq \h}}x^{k}y^{m-k}\ll \h^m\log^{\chi(m)}\h.
\]
{The proof of the second inequality is given in~\cite[eq. (3.2)]{Ku09}.}

It is known (see, e.g., \cite[Theorem 4.2.2]{Prasolov}) that if $P_1$ and $P_2$ are integral polynomials of degrees $n_1$ and $n_2$ respectively, then
\[
\left(2^{n_1+n_2-2}\sqrt{n_1+n_2+1}\right)^{-1}H(P_1)H(P_2)\leq H(P_1P_2),
\]
which implies
\[
\tilde r(\h)\leq \sum_{n_1+n_2=2m}s\left(2^{2m-2}\left(2m+1\right)^{1/2}\,\h\right)\ll \h^m\log^{\chi(m)}\h.
\]
Now to finish the proof it is sufficient to note that
\[
H(P)\ll h_\lambda(P)\ll H(P).
\]

\subsection{Proof of Lemma~\ref{000}}

Recall that  $A_l\subset\R^{m+1}$ is \tc{a set}  of points $(a_0,\ldots,a_m)\in \mathcal{E}_\bw$ such that the polynomial $P(z)=a_0z^{2m}+\ldots+a_mz^{m}+\ldots+a_0$ satisfies $\mu_P(\T_{\beta_1,\beta_2})=l$. For $z=e^{i\theta}$ we have
\[
P(z)=2e^{im\theta}\,\left(\sum_{k=1}^{m}a_{m-k} \cos\left(k\theta\right)+\frac{a_m}{2}\right)=:2e^{im\theta}\,\tilde P(\theta).
\]
Hence  $A_l$ is a set of points $(a_0,\ldots,a_m)\in \mathcal{E}_\bw$ such that $\mu_{\tilde P}([\beta_1,\beta_2])=l$.

{It is easy to see that} the boundary of $A_l$ is contained in the union of three sets:
\begin{enumerate}
	\item the boundary of $\mathcal{E}_\bw$;
	\item  the set
	\[
	A' = \left\{(a_0,\dots,a_m)\in\mathcal{E}_\bw:\tilde P(\beta_1)=0\quad\textrm{or}\quad\tilde P(\beta_2)=0\right\};
	\]
	\item the set $A''$ consisting of the points $(a_0,\dots,a_m)$ such that  the trigonometric polynomial $\tilde P$ has double real roots in $[\beta_1,\beta_2]$.
\end{enumerate}
Thus it is enough to show that each of these sets is of Lipschitz class.

\emph{(i) The boundary of $\mathcal{E}_\bw$.}	
Since $\mathcal{E}_\bw$ is a convex body, according to~\cite[Theorem~2.6]{mW12} its boundary  belongs to the Lipschitz class.

\emph{(ii) The set $A'$.} 
{Without loss of generality we can assume that} $P(\beta_1)=0$, which is equivalent to
\[
a_m=-2\sum_{k=1}^{m}a_{m-k} \cos\left(k\beta_1\right).
\]
Since $(a_0,\ldots,a_m)\in \mathcal{E}_\bw$, we have $a_0,\dots,a_{m-1}\leq C:=\max\limits_i\lambda_i$. Consider a Lipschitz map $\phi=(\phi_0,\dots,\phi_m):[0,1]^m\to\R^{m+1}$ defined  as 
\[
\phi_i(t_0,\dots,t_{m-1}):=Ct_i,\quad i=0,\dots,m-1,
\]
and
\[
\phi_m(t_0,\dots,t_{m-1}):=-2C\sum_{k=1}^{m}t_{m-k} \cos\left(k\beta_1\right).
\]
We obviously have
\[
a_i=\phi_i(a_0/C,\dots,a_{m-1}/C),\quad i=0,\dots,m-1,
\]
which implies $A'\subset\phi([0,1]^m)$. Therefore $A'$ is of Lipschitz class.

\emph{(iii) The set $A''$.}	
Suppose that $(a_0,\dots,a_m)\in A''$. Then
\[
\tilde{P}(\theta)=\sum_{k=1}^{m}a_{m-k} \cos\left(k\theta\right)+\frac{a_m}{2}
\]
has a multiple real root, say $\beta_0$, which implies
\begin{equation}\label{1441}
\tilde{P}(\beta_0)= 0, \quad
\tilde{P'}(\beta_0)= 0,
\end{equation}
or, excluding the trivial case $\beta_0=0$, equivalently,
\[
a_{m-1} = - \sum_{k=2}^m ka_{m-k}\frac{\sin(k\beta_0)}{\sin\beta_0},
\]
\begin{equation*}
a_m=-2\sum_{k=2}^{m}a_{m-k} \cos\left(k\beta_0\right)+2\cos\beta_0\sum_{k=2}^mka_{m-k}\frac{\sin(k\beta_0)}{\sin\beta_0}.
\end{equation*}
Again, $a_0,\dots,a_{m-1}\leq C:=\max\limits_i\lambda_i$. Moreover, we have $|\beta_0|\leq\pi$ (see~\eqref{2325}). Consider a map
$\phi=(\phi_0,\dots,\phi_{m}) : [0, 1]^{m} \to\R^{m+1} $
defined as 
\[
\phi_i(t,t_0,\dots,t_{m-2}):=Ct_i,\quad i=0,\dots,m-2,
\]
\[
\phi_{m-1}(t,t_0,\dots,t_{m-2}):= - C\sum_{k=2}^m kt_{m-k}\frac{\sin(k\pi t)}{\sin(\pi t)},
\]
and
\[
\phi_{m}(t,t_0,\dots,t_{m-2}):= -2C\sum_{k=2}^{m}t_{m-k} \cos\left(k\pi t\right)+2C\cos(\pi t)\sum_{k=2}^mkt_{m-k}\frac{\sin(k\pi t)}{\sin\pi t}.
\]
Since $\phi$ is continuously differentiable in a compact, it satisfies the Lipschitz condition. We obviously have
\[
a_i=\phi_i(\beta_0/\pi,a_0/C,\dots,a_{m-2}/C),\quad i=0,\dots,m,
\]
which implies $A''\subset\phi([0,1]^m)$. Therefore $A''$ is of Lipschitz class.

\section{Proofs of Corollaries}\label{0023}
Consider the kernel
\[
K_{\bw}(x,y)=\frac{\lambda_m^{2}}{4}+2\sum\limits_{k=1}^{m}\lambda_{m-k}^{2}\,\cos(kx)\cos(ky).
\]
We obviously have
\begin{multline}\label{1242}
p_{\bw}(t)=\left[\frac{\partial^2}{\partial x\partial y}\log K_\bw(x,y)\Big|_{x=y=t}\right]^{1/2}\\
=\left[\frac{K_{\bw}(t,t)\cdot\frac{\partial^2}{\partial x\partial y}K_{\bw}(x,y)\big|_{x=y=t}-\frac{\partial}{\partial x} K_{\bw}(x,t)\big|_{x=t}\cdot \frac{\partial}{\partial y} K_{\bw}(t,y)\big|_{y=t}}{K_{\bw}^2(t,t)}\right]^{1/2}.
\end{multline}

\subsection{Proof of Corollary \ref{0003}}
{In this case the kernel looks as follows:}
\begin{align*}
K_{\bw}(x,y)&=\frac12\binom{2m}{m}+\sum\limits_{k=1}^{m}\binom{2m}{m-k}\,\cos(kx)\cos(ky)\\
&=\frac12\binom{2m}{m}+\sum\limits_{k=1}^{m}\binom{2m}{m-k}\,\frac{e^{-ikx}+e^{ikx}}{2}\cdot\frac{e^{-iky}+e^{iky}}{2}\\
&=\frac{e^{-im(x+y)}}{4}\,\left(\sum\limits_{k=0}^{2m}{\binom{2m}{k}e^{ik(x+y)}}+\sum\limits_{k=0}^{2m}{\binom{2m}{k}e^{iky}e^{i(2m-k)x}}\right)\\
&=\frac{e^{-im(x+y)}}{4}\Big(\left(1+e^{i(x+y)}\right)^{2m}+\left(e^{iy}+e^{ix}\right)^{2m}\Big)\\
&=:\frac{e^{-im(x+y)}}{4}\tilde K_{\bw}(x,y).
\end{align*}
Since the factor $e^{-im(x+y)}/4$ does not affect the result, we obtain
\begin{align*}
p_{\bw}(t)&=\left[\frac{\partial^2}{\partial x\partial y}\log K_{\bw}(x,y)\Big|_{x=y=t}\right]^{1/2}=\left[\frac{\partial^2}{\partial x\partial y}\log \tilde K_{\bw}(x,y)\Big|_{x=y=t}\right]^{1/2}\\
&=\left[\frac{\tilde K_{\bw}(t,t)\cdot\frac{\partial^2}{\partial x\partial y}\tilde K_{\bw}(x,y)\big|_{x=y=t}-\frac{\partial}{\partial x}\tilde K_{\bw}(x,t)\big|_{x=t}\cdot \frac{\partial}{\partial y}\tilde K_{\bw}(t,y)\big|_{y=t}}{\tilde K_{\bw}^2(t,t)}\right]^{1/2}.
\end{align*}
The task now is to find the partial derivatives of $\tilde K_{\bw}(x,y)$ at $x=y=t$. We have
\begin{equation*}\label{eq17}
\tilde K_{\bw}(t,t)=\left(1+e^{2it}\right)^{2m}+2^{2m}e^{2imt}=2^{2m}e^{2imt}\,\left((\cos t)^{2m}+1\right),
\end{equation*}
\begin{align*}
\frac{\partial}{\partial x}\tilde K_{\bw}(x,t)\big|_{x=t}&=\frac{\partial}{\partial y}\tilde K_{\bw}(t,y)=2ime^{2it}\left(1+e^{2it}\right)^{2m-1}+2im\,2^{2m-1}e^{2imt}\\
&=2im\,2^{2m-1}e^{2imt}\,\left(e^{it}(\cos t)^{2m-1}+1\right),
\end{align*}
and
\begin{align*}\label{eq19}
\frac{\partial^2}{\partial x\partial y}&\tilde K_{\bw}(x,y)\big|_{x=y=t}\\
&=-2me^{2it}\left(1+e^{2it}\right)^{2m-1}-2m(2m-1)e^{4it}\left(1+e^{2it}\right)^{2m-2}
-2m(2m-1)\,2^{2m-2}e^{2imt}\\&=-2m\,2^{2m-2}e^{2imt}\,\Big(2e^{it}(\cos t)^{2m-1}
+(2m-1)e^{2it}(\cos t)^{2m-2}+(2m-1)\Big).
\end{align*}
Therefore,
\begin{align*}
p_{\bw}(t)&=\sqrt{\frac{m}{2}}\,\Big((\cos t)^{2m}+1\Big)^{-1}\cdot\Big(e^{2it}(\cos t)^{4m-2}-2e^{it}(\cos t)^{4m-1}-(2m-1)(\cos t)^{2m}\\
&-(2m-1)e^{2it}(\cos t)^{2m-2}+(2m-1)e^{it}(\cos t)^{2m-1}+1\Big)^{1/2}.
\end{align*}
Using the identities 
\begin{align*}
e^{it}&=\cos t + i\sin t,\\
e^{2it}&=\cos (2t) + i\sin (2t)=2(\cos t)^2-1+2i\sin t\cos t,
\end{align*}
we conclude with
\begin{align*}
p_{\bw}(t)&=\sqrt{\frac{m}{2}}\cdot\frac{\Big(1-(\cos^2 t)^{2m-1}+(2m-1)(\cos t)^{2m-2}\,\left(1-\cos^2 t\right)\Big)^{1/2}}{(\cos t)^{2m}+1}\\
&=\sqrt{\frac{m}{2}}\cdot\frac{\sqrt{1-\cos^2 t}\,\left(\sum\limits_{k=0}^{2m-2}\left(\cos t\right)^{2k}+(2m-1)\left(\cos t\right)^{2m-2}\right)^{1/2}}{\left(\cos t\right)^{2m}+1}\\
&=\sqrt{\frac{m}{2}}\cdot\frac{\left|\sin t\right|\,\left(\sum\limits_{k=0}^{2m-2}\left(\cos t\right)^{2k}+(2m-1)\left(\cos t\right)^{2m-2}\right)^{1/2}}{\left(\cos t\right)^{2m}+1}.
\end{align*}

\subsection{Proof of Corollary \ref{0004}}

Before we start recall some trigonometric formulas which will be used in our calculations:
\begin{align*}
\frac{\sin\left((N+\frac12)(x-y)\right)}{\sin\frac{x-y}{2}}&=1+2\sum\limits_{k=1}^{N}\cos(k(x-y));\\
\sin \left(\frac{N(x-y)}{2}\right) &=\sum\limits_{k\text{ odd}}(-1)^{(k-1)/2}\binom{N}{k}\left(\cos\frac{x-y}{2}\right)^{N-k}\left(\sin\frac{x-y}{2}\right)^{k};\\
\cos \left(\frac{N(x-y)}{2}\right) &=\sum\limits_{k\text{ even}}(-1)^{k/2}\binom{N}{k}\left(\cos\frac{x-y}{2}\right)^{N-k}\left(\sin\frac{x-y}{2}\right)^{k}.
\end{align*}

Consider the function $p_{\bw}(t)$ with weights $\bw=(1,\ldots,1)$. In this case the kernel has the form
\begin{align*}
K_{\bw}(x,y)&=\frac{1}{2}+\sum\limits_{k=1}^{m}\cos(kx)\cos(ky)=\frac{1}{2}+\frac12\sum\limits_{k=1}^{m}\cos k(x+y)+\frac12\sum\limits_{k=1}^{m}\cos k(x-y)\\
&=\frac{\sin\left((m+\frac12)(x+y)\right)}{4\sin\frac{x+y}{2}}+\frac{\sin\left((m+\frac12)(x-y)\right)}{4\sin\frac{x-y}{2}}.
\end{align*}
In order to compute $p_{\bw}(t)$ let us find the partial derivatives of the kernel $K_{\bw,m}(x,y)$ at $x=y=t$. Recall that $b_m=2m+1$. Thus we have
\[
K_{\bw}(t,t)=\frac{\sin(b_mt)}{4\sin t}+\frac{\sin\left((m+\frac12)(x-y)\right)}{4\sin\frac{x-y}{2}}\Big|_{x=y=t}=\frac{b_m}{4}+\frac{\sin(b_mt)}{4\sin t},
\]
\begin{align*}
\frac{\partial}{\partial x}K_{\bw,m}(x,t)\Big|_{x=t}&=\frac{\partial}{\partial y}K_{\bw,m}(t,y)\Big|_{y=t}\\
&=\left(\frac{m\cos \frac{b_m(x-y)}{2})}{4\sin\frac{x-y}{2}}-\frac{\sin \frac{2m(x-y)}{2}}{8\sin^2\frac{x-y}{2}}\right)\Big|_{x=y=t}+\frac{m\cos(b_mt)}{4\sin t}-\frac{\sin (2mt)}{8\sin^2t}\\
&=\frac{m}{4\sin\frac{x-y}{2}}\Big|_{x=y=t}-\frac{m}{4\sin\frac{x-y}{2}}\Big|_{x=y=t}+\frac{m\cos(b_mt)}{4\sin t}-\frac{\sin (2mt)}{8\sin^2t}\\
&=\frac{m\cos(b_mt)}{4\sin t}-\frac{\sin (2mt)}{8\sin^2t},
\end{align*}
and
\begin{align*}
\frac{\partial^2}{\partial x\partial y}K_{\bw,m}(x,y)\Big|_{x=y=t}&=\left(\frac{m^2\sin \frac{b_m(x-y)}{2})}{4\sin\frac{x-y}{2}}+\frac{m\cos \frac{2m(x-y)}{2})}{4\sin^2\frac{x-y}{2}}-\frac{\sin \frac{2m(x-y)}{2})}{8\sin^3\frac{x-y}{2}}\right)\Big|_{x=y=t}\\
&-\frac{m^2\sin(b_mt)}{4\sin t}-\frac{m\cos (2mt)}{4\sin^2t}+\frac{\cos t\, \sin (2mt)}{8\sin^3t}\\
&=\frac{m^2b_m}{4}+\frac{m}{4\sin^2\frac{x-y}{2}}\Big|_{x=y=t}-\frac{m^2(2m-1)}{4}-\frac{m}{4\sin^2\frac{x-y}{2}}\Big|_{x=y=t}\\
&+\frac{m(2m-1)(m-1)}{2}-\frac{m^2\sin(b_mt)}{4\sin t}-\frac{m\cos (2mt)}{4\sin^2t}+\frac{\cos t\, \sin (2mt)}{8\sin^3t}\\
&=\frac{\cos t\, \sin (2mt)}{8\sin^3t}-\frac{m^2\sin(b_mt)}{4\sin t}-\frac{m\cos (2mt)}{4\sin^2t}+\frac{(m^2+m)b_m}{12}.
\end{align*}
Substituting this into~\eqref{1242} finishes the proof.

\bibliographystyle{plainnat}

\end{document}